\documentclass[12pt]{amsart}

\usepackage{amsmath,amsfonts,amssymb}

\usepackage{color}
\usepackage{xcolor}
\usepackage{hyperref}
\newtheorem{defn}{Definition}[section]
\newtheorem{theo}[defn]{Theorem}
\newtheorem{lem}[defn]{Lemma}
\newtheorem{prop}[defn]{Proposition}
\newtheorem{cor}[defn]{Corollary}

\DeclareMathOperator{\cpct}{Cap}
\DeclareMathOperator{\ccpct}{cap}
\DeclareMathOperator{\lin}{span}

\begin{document}


\title[Capacities, removable sets and $L^p$-uniqueness]{Capacities, removable sets and $L^p$-uniqueness on Wiener spaces}

\author{Michael Hinz$^1$ \and Seunghyun Kang$^2$}

\address{$^1$Fakult\"at f\"ur Mathematik, Universit\"at Bielefeld, 
 33501 Bielefeld, Germany}
 \email{mhinz@math.uni-bielefeld.de}
 
 \address{$^2$Department of Mathematical Sciences, Seoul National University, GwanAkRo 1, 
 Gwanak-Gu, Seoul 08826, Korea }
\email{power4454@snu.ac.kr}

\thanks{$(^{1,2})$ Research supported by the DFG IRTG 2235: 'Searching for the regular in the irregular: Analysis of singular and random systems'.}

\begin{abstract}
We prove the equivalence of two different types of capacities in abstract Wiener spaces. This yields 
a criterion for the $L^p$-uniqueness of the Ornstein-Uhlenbeck operator and its integer powers defined on suitable algebras of functions vanishing in a neighborhood of a given closed set $\Sigma$ of zero Gaussian measure. To prove the equivalence we show the $W^{r,p}(B,\mu)$-boundedness of certain smooth nonlinear truncation operators acting on potentials of nonnegative functions. We also give connections to Gaussian Hausdorff measures. Roughly speaking, if $L^p$-uniqueness holds then the 'removed' set $\Sigma$ must have sufficiently large codimension, in the case of the Ornstein-Uhlenbeck operator for instance at least $2p$.
\end{abstract}
\maketitle
\tableofcontents


\section{Introduction} 

The present article deals with capacities associated with Ornstein-Uhlenbeck operators on abstract Wiener spaces $(B,\mu,H)$, \cite{B1, Bo, Gross67, IkWa89, LMPnotes, M, Ma97, S, Su88}, and applications to $L^p$-uniqueness problems for Ornstein-Uhlenbeck operators and their integer powers, endowed with algebras of functions vanishing in a neighborhood of a small closed set.

Our original motivation comes from $L^p$-uniqueness problems for operators $L$ endowed with a suitable algebra $\mathcal{A}$ of functions, the special case $p=2$ is the problem of essential self-adjointness. For the 'globally defined' operator $L$ on the entire space $L^p$-uniqueness is well understood, see for instance \cite{E} and the references cited there. If the globally defined operator is $L^p$-unique one can ask whether the removal of a small set (or, in other words, the introduction of a small boundary) destroys this uniqueness or not. A loss of uniqueness means that extensions to generators of $C_0$-semigroups, \cite{Na}, with different boundary conditions exist. The answer to this question depends on the size of the removed set. The most classical example may be the essential self-adjointness problem for the Laplacian $\Delta$ on $\mathbb{R}^n$, endowed with the algebra $C_c^\infty(\mathbb{R}^n\setminus \left\{0\right\})$ of smooth compactly supported functions on $\mathbb{R}^n$ with the origin $\left\{0\right\}$ removed. It is well known that this operator is essential self-adjoint in $L^2(\mathbb{R}^n)$ if and only if $n\geq 4$, \cite[p.114]{Takeda92} and \cite[Theorem X.11, p.161]{R}. Generalizations of this example to manifolds have been provided in \cite{CdV82} and \cite{Masamune99}, more general examples on Euclidean spaces can be found in \cite{AGMST13} and \cite{HiKaMa}, further generalizations to manifolds and metric measure spaces will be discussed in \cite{HiMa}. For the Laplacian on $\mathbb{R}^n$ one main observation is that, 
if a compact set $\Sigma$ of zero measure is removed from $\mathbb{R}^n$, the essential self-adjointness of $(\Delta, C_c^\infty(\mathbb{R}^n\setminus \Sigma))$ in $L^2(\mathbb{R}^n)$ implies that $\dim_H \Sigma\leq n-4$, where $\dim_H$ denotes the Hausdorff dimension. See \cite[Theorems 10.3 and 10.5]{AGMST13} or \cite[Theorem 2]{HiKaMa}. This necessary 'codimension four' condition can be rephrased by saying that we must have $\mathcal{H}^{n-d}(\Sigma)=0$ for all $d<4$, where $\mathcal{H}^{n-d}$ denotes the Hausdorff measure of dimension $n-d$. 

Having in mind coefficient regularity or boundary value problems for operators in infinite dimensional spaces, see e.g. \cite{B2, DPL, DPL2, H, H1}, one may wonder whether a similar 'codimension four' condition can be observed in infinite dimensional situations. For the case of Ornstein-Uhlenbeck operators on abstract Wiener spaces an affirmative answer to this question follows from the present results in the special case $p=2$.

The basic tools to describe the critical size of a removed set $\Sigma\subset B$ are capacities associated with the Sobolev spaces $W^{r,p}(B,\mu)$ for the $H$-derivative respectively the Ornstein-Uhlenbeck semigroup, \cite{B1, Bo, Gross67, IkWa89, LMPnotes, M, Ma97, S, Su88}. Such capacities can be introduced following usual concepts of potential theory, \cite{Bo, FdP91, Ma97, SI, S, So93, So93b, Su88}, see Definition \ref{def5} below, and they are known to be connected to Gaussian Hausdorff measures, \cite{FdP92}. Uniqueness problems connect easier to another, slightly different definition of capacities, where the functions taken into account in the definition are recruited from the initial algebra $\mathcal{A}$ and, roughly speaking, are required to be equal to one on the set in question, see Definitions \ref{def6} and \ref{def7}.
This type of definition connects them to an algebraic ideal property which is helpful to investigate extensions of operators initially defined on ideals of $\mathcal{A}$. For Euclidean Sobolev spaces these two types of capacities are known to be equivalent, see for instance \cite[Section 2.7]{A}. The proofs of these equivalences go back to Mazja, Khavin, Adams, Hedberg, Polking and others, \cite{A76, A, AP, Mazja72, Mazja, MazjaKhavin}, and rely on bounds in Sobolev norms for certain nonlinear composition operators acting on the cone of nonnegative Sobolev functions, see e.g. \cite[Theorem 3]{A76}, or the cone of potentials of nonnegative functions, see e.g. \cite[Theorem 2]{A76} or \cite[Theorem 3.3.3]{A}. Apart from the first order case $r=1$ this is nontrivial, because in finite dimensions Sobolev spaces are not stable under such compositions, see for instance \cite[Theorem 3.3.2]{A}. Apart from the case $p=2$, where one can also use an integration by parts argument, \cite[Theorem 3]{A76}, the desired bounds are shown using suitable Gagliardo-Nirenberg inequalities, \cite{AP, Mazja72}, or suitable multiplicative estimates of Riesz or Bessel potential operators involving Hardy-Littelwood maximal functions and the $L^p$-boundedness of the latter, \cite[Theorem 1.1.1, Proposition 3.1.8]{A} The constants in these estimates are dimension dependent. 
 
Sobolev spaces $W^{r,p}(B,\mu)$ over abstract Wiener spaces $(B,\mu,H)$ are stable under compositions with bounded smooth functions, \cite[Remark 5.2.1 (i)]{B1}, but one still needs to establish quantitative bounds. We establish Sobolev norm bounds for nonlinear composition operators acting on 
potentials of nonnegative functions, Lemma \ref{lemma1}. To obtain it, we use the $L^p$-boundedness of the maximal function in the sense of Rota and Stein for the Ornstein-Uhlenbeck semigroup, \cite[Theorem 3.3]{S}, this provides a similar multiplicative estimate as in the finite dimensional case, see Lemma \ref{lemma_multest}. From the Sobolev norm estimate for compositions we can then deduce the desired equivalence of capacities, Theorem \ref{T:equiv}, where $\mathcal{A}$ is chosen to be the set of smooth cylindrical functions or the space of Watanabe test functions. Applications of this equivalence provide $L^p$-uniqueness results for the Ornstein-Uhlenbeck operator and, under a sufficient condition that ensures they generate $C_0$-semigroups, also for its integer powers, see Theorem \ref{T:Lp}. In particular, if $\Sigma\subset B$ is a given closed set of zero Gaussian measure, then the Ornstein-Uhlenbeck operator, endowed with the algebra of cylindrical functions vanishing in a neighborhood of $\Sigma$ (or the algebra 
of Watanabe test functions vanishing q.s. on a neighborhood of $\Sigma$) is $L^p$-unique if and only if the $(2,p)$-capacity of $\Sigma$ is zero, see Theorem \ref{T:Lp}. Combined with results from \cite{FdP92} on Gaussian Hausdorff measures, we then observe that the $L^p$-uniqueness of this Ornstein-Uhlenbeck operator 'after the removal of $\Sigma$'
implies that the Gaussian Hausdorff measure $\varrho_d(\Sigma)$ of codimension $d$ of $\Sigma$ must be zero for all $d<2p$, see Corollary \ref{C:codimension}. In particular, if the operator is essentially self-adjoint on $L^2(B,\mu)$, then $\varrho_d(\Sigma)$ must be zero for all $d<4$, what is an analog of the necessary 'codimension four' condition knwon from the Euclidean case. 

In the next section we recall standard items from the analysis on abstract Wiener spaces. In Section \ref{S:caps} we define Sobolev capacities and prove their equivalence, based on the norm bound on nonlinear compositions, which is proved in Section \ref{S:truncationproof}. Section \ref{S:Lp} contains the mentioned $L^p$-uniqueness results.
The connection to Gaussian Hausdorff measures is briefly discussed in Section \ref{S:Hausdorff}, followed by some remarks on related Kakutani theorems for multiparameter processes in Section \ref{S:processes}.

\section*{Acknowledgements}

The authors would like to thank Masanori Hino, Jun Masamune, Michael R\"ockner and Gerald Trutnau for inspiring and helpful discussions.

\section{Preliminaries}

Following the presentation in \cite{S}, we provide some basic definitions and facts.

Let $(B,\mu, H)$ be an \emph{abstract Wiener space}. That is, $B$ is a real separable Banach space, $H$ is a real separable Hilbert space which is embedded densely and continuously on $B$, and $\mu$ is a Gaussian measure on $B$ with
$$\int_B{\exp\{\sqrt{-1}\langle \varphi, y\rangle\}\mu(dy)}=\exp\{-\dfrac{1}{2}|\varphi|^2_{H^*}\}, \quad \varphi\in B^*,$$
see for instance \cite[Definition 1.2]{S}. Here we identify $H^\ast$ with $H$ as usual, so that $B^\ast\subset H\subset B$.  Since every $\varphi\in B^\ast$ is $N(0,\left\|\varphi\right\|_{H}^2)$-distributed, it is an element of $L^2(B;\mu)$ and the map $\varphi\mapsto \left\langle\varphi,\cdot\right\rangle$ is an isometry from $B^\ast$, equipped with the scalar product $\left\langle\cdot,\cdot\right\rangle_H$, into $L^2(B,\mu)$. It extends uniquely to an isometry
\begin{equation}\label{E:isometry}
h\mapsto \hat{h}
\end{equation} 
from $H$ into $L^2(B,\mu)$. A function $f:B\to\mathbb{R}$ is said to be \emph{$H$-differentiable at $x\in B$} if there exists some $h^\ast \in H^\ast$ such that 
\[\frac{d}{dt} f(x+th)|_{t=0}=\left\langle h,h^\ast\right\rangle\]
for all $h\in H$. If $f$ is $H$-differentiable at $x$ then $h^\ast$ is uniquely determined, denoted by $Df(x)$ and refereed to as the \emph{$H$-derivative of $f$ at $x$}. See \cite[Definition 2.6]{S}. For a function $f$ that is $H$-differentiable at $x\in B$ and an element $h$ of $H$ we can define the \emph{directional derivative $\partial_h f(x)$ of $f$ at $x$} by
\[\partial_h f(x):=\left\langle Df(x),h\right\rangle_H.\]
A function $f:B\to\mathbb{R}$ is said to be \emph{$k$-times $H$-differentiable at $x\in B$} if there exists a continuous $k$-linear mapping $\Phi_x:H^k\to\mathbb{R}$ such that 
\[\frac{\partial^k}{\partial t_1\cdots \partial t_k} f(x+t_1h_1+\dots+t_kh_k)|_{t_1=\dots=t_k=0}=\Phi_x(h_1,\dots h_k)\]
for all $h_1,\dots, h_k\in H$. If so, $\Phi_x$ is unique and denoted by $D^kf(x)$. 
A function $f:B\to\mathbb{R}$ is called a \emph{(smooth) cylindrical function} if there exist an integer $n\geq 1$, linear functionals $l_1,...,l_n\in B^\ast$ and a function $F\in C_b^\infty(\mathbb{R}^n)$ such that 
\begin{equation}\label{E:cylindricalf}
f=F(l_1,...,l_n).
\end{equation}
The space of all such cylindrical functions on $B$ we denote by $\mathcal{F}C_b^\infty$. Clearly $\mathcal{F}C_b^\infty$ is an algebra under pointwise multiplication and stable under the composition with functions $T\in C_b^\infty(\mathbb{R})$.

A cylindrical function $f\in \mathcal{F}C_b^\infty$ as in (\ref{E:cylindricalf}) is infinitely many times $H$-differentiable at any $x\in B$, and for any $k\geq 1$ we have 
\[D^kf(x)=\sum_{j_1,\dots j_k=1}^\infty \partial_{j_1}\cdots\partial_{j_k}F(\left\langle x,l_1\right\rangle, ..., \left\langle x, l_n\right\rangle)\:l_{j_1}\otimes\cdots\otimes l_{j_k},\]
where $\partial_j$ denotes the $j$-th partial differentiation in the Euclidean sense. The space $\mathcal{F}C^{\infty}_b$ is dense in $L^p(B,\mu)$ for any $1\leq p<+\infty$, see e.g. \cite[Lemma 2.1]{B}. 

We write $\mathcal{H}_0:=\mathbb{R}$, $\mathcal{H}_1:=H$ and generalizing this, denote by $\mathcal{H}_k$ the space of $k$-linear maps $A:H^k\to \mathbb{R}$ such that 
\begin{equation}\label{E:HS}
\left\|A\right\|_{\mathcal{H}_k}^2:=\sum_{j_1,\dots, j_k=1}^\infty (A(e_{j_1},\dots, e_{j_k}))^2<+\infty,
\end{equation}
where $(e_i)_{i=1}^\infty$ is an orthonormal basis in $H$. The value of this norm does not depend on the choice of this basis. See \cite[p.3]{B1a}. Clearly every such $k$-linear map $A$ can also be seen as a linear map $A:H^{\otimes k}\to\mathbb{R}$, where $H^{\otimes k}$ denotes the $k$-fold tensor product of $H$, with this interpretation we have  $A(e_{j_1}\otimes ...\otimes  e_{j_k})=A(e_{j_1},\dots, e_{j_k})$ and by (\ref{E:HS}) the operator $A$ is a Hilbert-Schmidt operator. For later use we record the following fact.

\begin{prop}\label{P:HSnormestimate} For any $A\in\mathcal{H}_k$ we have 
\begin{multline}
\|A\|_{\mathcal{H}_k}\leq 2k^k\sup\left\lbrace |A(h_1,...,h_k)|: \text{ $h_1,\cdots, h_k$ are members} \right.\notag\\
\left.\text{of an orthonormal system in $H$, not necessarily distinct}\right\rbrace.
\end{multline}
\end{prop}

\begin{proof}
By Parseval's identity and Cauchy-Schwarz in $H^{\otimes k}$ we have 
\[\left\|A\right\|_{\mathcal{H}_k}=\sup\left\lbrace |Ay|: \text{ $y\in H^{\otimes k}$ and $\left\|y\right\|_{H^{\otimes k}}=1$}\right\rbrace.\] 
Choose an element $y=y_1 \otimes ...\otimes y_k\in H^{\otimes k}$ such that $\|y\|_{H^{\otimes k}}=1$ and $\|A\|_{\mathcal{H}^k}\leq 2|Ay|$. Without loss of generality we may assume that $\|y_j\|_H=1$, $1\leq j\leq k$. Choosing an orthonormal basis $(b_i)_{i=1}^n$ in the subspace $\lin \left\lbrace y_1,...,y_k\right\rbrace$ of $H$ we observe $n\leq k$ and 
$y_j=\sum_{i=1}^n b_i\lambda_{ij}$ with some $|\lambda_{ij}|\leq 1$. Since this implies  
\[|Ay|\leq \sum_{i_1,\cdots,i_k\in \{1,\cdots,n\}}{|A(b_{i_1},\cdots,b_{i_k}})|,\] 
we obtain the desired result.
\end{proof}

We recall the definition of Sobolev spaces on $B$. For any $1\leq p<+\infty$ and $k\geq 0$ let $L^p(B,\mu,\mathcal{H}_k)$ denote the $L^p$-space of functions from $B$ into $\mathcal{H}_k$. For any $1\leq p<+\infty$ and integer $r\geq 0$ set 
\begin{equation}\label{E:Sobonorm}
\left\|f\right\|_{W^{r,p}(B,\mu)}:=\sum_{k=0}^r \left\|D^k f\right\|_{L^p(B,\mu,\mathcal{H}_k)},
\end{equation}
$f\in\mathcal{F}C_b^\infty$. The \emph{Sobolev class $W^{r,p}(B,\mu)$} is defined as the completion of $\mathcal{F}C_b^\infty$ in this norm, see \cite[Section 5.2]{B1} or \cite[Section 8.1]{B1a}. In particular, $W^{0,p}(B,\mu)=L^p(B,\mu)$.
For $f\in W^{r,p}(B,\mu)$ the derivatives $D^kf$, $k\leq r$, are well defined as elements of $L^p(B,\mu)$, see \cite[Section 5.2]{B1}. By definition the spaces $W^{r,p}(B,\mu)$ are Banach spaces, Hilbert if $p=2$. The space $W^{\infty}$ of \emph{Watanabe test functions} is defined as
\[W^\infty:=\bigcap_{r\geq 1,\: 1\leq p<+\infty} W^{r,p}(B,\mu).\]
We have $\mathcal{F}C_b^\infty\subset W^\infty$, in particular, $W^\infty$ is a dense subset of every $L^p(B,\mu)$ and $W^{r,p}(B,\mu)$. 

In contrast to Sobolev spaces over finite dimensional spaces, \cite[Theorem 3.3.2]{A}, also the Sobolev classes $W^{r,p}(B,\mu)$, $r\geq 2$, are known to be stable under compositions $u\mapsto T(u)=T\circ u$ with functions $T\in C_b^\infty(\mathbb{R})$, as follows from the evaluation of an integration by parts identity together with the chain rule, applied to cylindrical functions. See \cite[Remark 5.2.1 (i)]{B1} or \cite[Proposition 8.7.5]{B1a}. In particular, the space $W^\infty$ is stable under compositions with functions from $C_b^\infty(\mathbb{R})$. Also, it is an algebra with respect to pointwise multiplication, \cite[Corollary 5.8]{Ma97}.

Given a bounded (or nonnegative) Borel function $f:B\to\mathbb{R}$ and $t>0$ set 
\begin{equation}\label{E:Mehler} 
P_tf(x):=\int_B{f(e^{-t}x+\sqrt{1-e^{-2t}}y)\mu(dy)},\quad x\in B.
\end{equation} 
The function $P_tf$ is again bounded (resp. nonnegative) Borel on $B$ and the operators $P_t$ form a semigroup, i.e. that for any $s,t>0$ we have $P_{t+s}=P_tP_s$. The semigroup $(P_t)_{t>0}$ is called the \emph{Ornstein-Uhlenbeck semigroup on $B$}. For any $1\leq p\leq +\infty$ it extends to a contraction semigroup $(P^{(p)}_t)_{t>0}$ on $L^p(B,\mu)$, \cite[Proposition 2.4]{S}, strongly continuous for $1\leq p<+\infty$. The semigroup $(P^{(2)}_t)_{t>0}$ is a sub-Markovian symmetric semigroup on $L^2(B,\mu)$ in the sense of \cite[Definition I.2.4.1]{Bo}. The infinitesimal generators $(\mathcal{L}^{(p)},\mathcal{D}(\mathcal{L}^{(p)}))$ of $(P^{(p)}_t)_{t>0}$ is called the \emph{Ornstein-Uhlenbeck operator on $L^p(B,\mu)$}, \cite[Section 2.1.4]{S}. We will always write $P_t$ and $\mathcal{L}$ instead of $P^{(p)}_t$ and $\mathcal{L}^{(p)}$, the meaning will be clear from the context. Given $r>0$ and a bounded (or nonnegative)  Borel function $f:B\to\mathbb{R}$, set 
\begin{equation}\label{E:Vr}
V_{r}f:=\dfrac{1}{\Gamma(r/2)}\int^{\infty}_0{t^{r/2-1}e^{-t}P_t fdt},
\end{equation}
where $\Gamma$ denotes the Euler Gamma function. The function $V_r f$ is again bounded (resp. nonnegative) Borel, and for any $1\leq p<\infty$ the operators $V_r$ form a strongly continuous contraction semigroup $(V_r)_{r>0}$ on $L^p(B,\mu)$, see \cite[Corollary 5.3.3]{B1} or \cite[Proposition 4.7]{S}, symmetric for $p=2$. In any of these spaces the operators $V_r$ are the powers $(I-\mathcal{L})^{-r/2}$ of order $r/2$ of the respective $1$-resolvent operators $(I-\mathcal{L})^{-1}$.
Meyer's equivalence, \cite[Theorem 8.5.2]{B1a}, \cite[Theorem 4.4]{S}, states that for any integer $r\geq 1$ and any $1<p<+\infty$ and any $u\in W^{r,p}(B,\mu)$ we have
\begin{equation}\label{E:equivnorm}
c_1\left\|u\right\|_{W^{r,p}(B,\mu)}\leq \left\| (I-\mathcal{L})^{r/2} u\right\|_{L^p(B,\mu)}\leq c_2\left\|u\right\|_{W^{r,p}(B,\mu)}
\end{equation}
with constants $c_1>0$ and $c_2>0$ depending only on $r$ and $p$. By the continuity of the $V_r$ and the density of cylindrical functions we observe $W^{r,p}(B,\mu)=V_r(L^p(B,\mu))$. The operator $V_r$ acts as an isometry from $W^{s,p}(B,\mu)$ onto $W^{s+r,p}(B,\mu)$, \cite[Chapter II, Theorem 7.3.1]{Bo}. For later use we record the following well known fact.

\begin{prop}\label{P:preserve}
For any $r>0$ we have $V_r(\mathcal{F}C_b^\infty)\subset \mathcal{F}C_b^\infty$ and $V_r(W^\infty)\subset W^\infty$.
\end{prop}
\begin{proof}
From the preceding lines it is immediate that $V_r(W^\infty)\subset W^\infty$. To see the remaining statement 
suppose $f\in \mathcal{F}C_b^\infty$ with $f=F(l_1,...,l_n)$, $l_i\in B^\ast$, $F\in C_b^\infty(\mathbb{R}^n)$, and by applying Gram-Schmidt we may assume $\left\lbrace l_1,...,l_n\right\rbrace$ is an orthonormal system in $H$. The Ornstein-Uhlenbeck semigroup $(T_t^{(n)})_{t>0}$ on $L^2(\mathbb{R}^n)$, defined by
\[T_t^{(n)}F(\xi)=\int_{\mathbb{R}^n} F(e^{-t}\xi+\sqrt{1-e^{-2t}}\eta)(2\pi)^{-n/2}\:e^{-|\eta|^2/2}d\eta,\]
preserves smoothness, i.e. $T_t^{(n)}F\in C_b^\infty(\mathbb{R}^n)$ for any $F\in C_b^\infty(\mathbb{R}^n)$.

Given $x\in B$ and writing $\xi=(\left\langle x, l_1\right\rangle_H,...,\left\langle x,l_n\right\rangle_H)$, we have 
\begin{align}
&P_tf(x)\notag\\
&=\int_B F(\big\langle e^{-t}x+\sqrt{1-e^{-2t}}y,l_1\big\rangle_H,...,\big\langle e^{-t}x+\sqrt{1-e^{-2t}}y,l_n\big\rangle_H)\mu(dy)\notag\\
&= \int_B F(e^{-t}\xi+\sqrt{1-e^{-2t}}(\left\langle y, l_1\right\rangle_H,...,\left\langle y,l_n\right\rangle_H))\mu(dy)\notag\\
&=\int_{\mathbb{R}^n} F(e^{-t}\xi+\sqrt{1-e^{-2t}}\eta)(2\pi)^{-n/2}\:e^{-|\eta|^2/2}d\eta\notag\\
&=F_t^{(n)}(\left\langle x, l_1\right\rangle_H,...,\left\langle x,l_n\right\rangle_H),\notag
\end{align}
where $F_t^{(n)}=T_t^{(n)}F$. Consequently $P_tf\in \mathcal{F}C_b^\infty$, and using (\ref{E:Vr}) and dominated convergence it follows that $V_rf\in\mathcal{F}C_b^\infty$.
\end{proof}

Although different in nature both $\mathcal{F}C_b^\infty$ and $W^\infty$ can serve as natural replacements in infinite dimensions for algebras of smooth differentiable functions in Euclidean spaces or on manifolds.


\section{Capacities and their equivalence}\label{S:caps}

We define two types of capacities related to $W^{r,p}(B,\mu)$-spaces and verify their equivalence. 

The following definition is standard, see for instance \cite{FdP91, SI}.

\begin{defn}\label{def5} Let $1\leq p<+\infty$ and let $r>0$ be an integer. For open $U\subset B$, let  
$$\cpct_{r,p}(U):=\inf\{\left\|f\right\|^p_{L^p}|\:f\in L^p(B,\mu),\, V_r f\geq 1 \text{ $\mu$-a.e. on } U\}$$
and for arbitrary $A\subset B$,
$$\cpct_{r,p}(A):=\inf\{\cpct_{r,p}(U)|\:A\subset U,\ \text{$U$ open}\}.$$
\end{defn}

We give two further definitions of $(r,p)$-capacities. The first one is based on cylindrical functions and resembles \cite[Definition 2.7.1]{A} and \cite[Chapter 13]{Mazja}. 

\begin{defn}\label{def6} Let $1\leq p<+\infty$ and let $r>0$ be an integer. For an open set $U\subset B$ define
\[\ccpct_{r,p}^{(\mathcal{F}C_b^\infty)} (U):=\inf\left\lbrace \left\|u\right\|_{W^{r,p}(B,\mu)}^p|\: u\in \mathcal{F}C_b^\infty, \: \text{$u=1$ on $U$}\right\rbrace,\] 
and for an arbitrary set $A\subset B$,
$$\ccpct_{r,p}^{(\mathcal{F}C_b^\infty)}(A):=\inf\{\ccpct_{r,p}^{(\mathcal{F}C_b^\infty)}(U)|\: A\subset U,\, \text{$U$ open}\}.$$
\end{defn}

The capacities $\ccpct_{r,p}^{(\mathcal{F}C_b^\infty)}$ have useful 'algebraic' properties which we will use in Section \ref{S:Lp}.

One can give a similar definition based on the space $W^\infty$. To do so, we recall some potential theoretic notions.
If a property holds outside a set $E\subset B$ with $\cpct_{r,p}(E)=0$ then we say it holds \emph{$(r,p)$-quasi everywhere (q.e.)}. We follow \cite[Chapter IV, Section 1.2]{Ma97} and call a set $E\subset B$ \emph{slim} if  $\cpct_{r,p}(E)=0$ for all $1< p<+\infty$ and all integer $r>0$, and if a property holds outside a slim set, we say it holds \emph{quasi surely (q.s.)}. A function $u:B\to \mathbb{R}$ is said to be \emph{$(r,p)$-quasi continuous} if for any $\varepsilon>0$ we can find an open set $U_\varepsilon\subset B$ such that $\cpct_{r,p}(U)<\varepsilon$ and the restriction $u|_{U_\varepsilon^c}$ of $u$ to $U_\varepsilon^c$ is continuous. Every function $u\in W^{r,p}(B,\mu)$ admits a $(r,p)$-quasi-continuous version $\widetilde{u}$, unique in the sense that two different quasi continuous versions can differ only on a set of zero $(r,p)$-capacity. Since continuous functions are dense in $W^{r,p}(B,\mu)$ this follows by standard arguments, see for instance \cite[Chapter I, Section 8.2]{Bo}. Now one can follow \cite[Chapter IV, Section 2.4]{Ma97} to see that for any $u\in W^\infty$ there exists a function $\widetilde{u}:B\to\mathbb{R}$ such that $u=\widetilde{u}$ $\mu$-a.e. and for all $r$ and $p$ the function $\widetilde{u}$ is $(r,p)$-quasi continuous. It is referred to as the \emph{quasi-sure redefinition} of $u$ and it is unique in the sense that the difference of two quasi-sure redefinitions of $u$ is zero $(r,p)$-quasi everywhere for all $r$ and $p$, \cite{Ma97}.

\begin{defn}\label{def7} Let $1\leq p<+\infty$ and let $r>0$ be and integer. For an open set $U\subset B$ define
\[\ccpct_{r,p}^{(W^\infty)}(U):=\inf\left\lbrace \left\|u\right\|_{W^{r,p}(B,\mu)}^p|\: u\in W^\infty, \: \text{$\widetilde{u}=1$ on $U$ q.s.}\right\rbrace,\] 
where $\widetilde{u}$ denotes the quasi-sure redefinition of $u$ with respect to the capacities from Definition \ref{def5}, and for an arbitrary set $A\subset B$,
$$\ccpct_{r,p}^{(W^\infty)}(A):=\inf\{\ccpct_{r,p}^{(W^\infty)}(U)|\: A\subset U,\, \text{$U$ open}\}.$$
\end{defn}
This definition may seem a bit odd because it refers to Definition \ref{def5}. However, for some applications capacities based on the algebra $W^\infty$ may be more suitable that those based on cylcindrical functions.

The following equivalence can be observed.

\begin{theo}\label{T:equiv}
Let $1< p<+\infty$ and let $r>0$ be an integer. Then there are positive constants $c_3$ and $c_4$ depending only on $p$ and $r$ such that for any set $A\subset B$ we have 
\begin{equation}\label{E:FCbinfty}
c_3\:\ccpct_{r,p}^{(\mathcal{F}C_b^\infty)}(A)\leq \cpct_{r,p}(A)\leq c_4\:\ccpct_{r,p}^{(\mathcal{F}C_b^\infty)}(A)\end{equation}
and 
\begin{equation}\label{E:Winfty}
c_3\:\ccpct_{r,p}^{(W^\infty)}(A)\leq \cpct_{r,p}(A)\leq c_4\:\ccpct_{r,p}^{(W^\infty)}(A).
\end{equation}
\end{theo}

Theorem \ref{T:equiv} is an analogue of corresponding results in finite dimensions, \cite[Theorem A]{AP}, \cite[Theorem 3.3]{Mazja72}, see also \cite[Section 2.7 and Corollary 3.3.4]{A} or \cite[Sections 13.3 and 13.4]{Mazja}.

One ingredient of our proof of Theorem \ref{T:equiv} is a bound in $W^{r,p}(B,\mu)$-norm for compositions with suitable smooth truncation functions. For the spaces $W^{1,p}(B,\mu)$ such a bound is clear from the chain rule for $D$ respectively from general Dirichlet form theory, see \cite{Bo}. Norm estimates in $W^{r,p}(B,\mu)$ for compositions $T\circ u$  of elements $u\in W^{r,p}(B,\mu)$ with suitable smooth functions $T:\mathbb{R}\to\mathbb{R}$ can be obtained via the chain rule. For instance, in the special case $r=2$ the chain and product rules and the definition of the generator $\mathcal{L}$ imply
\[\mathcal{L} T(u)= T'(u)\mathcal{L}u + T''(u)\left\langle Du, Du\right\rangle_H, \quad u\in W^{2,p}.\]
By (\ref{E:equivnorm}) it would now suffice to show a suitable bound for $\mathcal{L} T(u)$ in $L^p$, and the summand more difficult to handle is the one involving the first derivatives $Du$. In the finite dimensional Euclidean case an $L^p$-estimate for it follows immediately from a simple integration by parts argument, \cite[Theorem 3]{A76}, or by a use of a suitable Gagliardo-Nirenberg inequality, \cite{AP, Mazja72}. Integration by parts for Gaussian measures comes with an additional 'boundary' term involving the direction $h\in H$ of differentiation that spoiles the original trick, and the classical proof of the Gagliardo-Nirenberg inequality involves dimension dependent constants. A simple alternative approach, suitable for any integer $r>0$, is to prove truncation results for potentials in a similar way as in \cite[Theorem 3.3.3]{A}, so that a quick evaluation of the first order term above follows from estimates in terms of the maximal function, \cite[Proposition 3.1.8]{A}. This method can be made dimension independent if the Hardy-Littlewood maximal function is replaced by the maximal function in terms of the semigroup operators (\ref{E:Mehler}) in the sense of Rota and Stein, \cite[Theorem 3.3]{S}, \cite[Chapter III, Section 3]{St70}, see Lemma \ref{lemma_multest} below. We obtain the following variant of a Theorem due to Mazja and Adams, \cite[Theorems 2 and 3]{A76}, \cite[Theorem 3.3.3]{A}, now for Sobolev spaces $W^{r,p}(B,\mu)$ over abstract Wiener spaces.  A proof will be given in Section \ref{S:truncationproof} below.

\begin{lem}\label{lemma1} Assume $1<p<+\infty$ and let $r>0$ be an integer. Let $T\in C^\infty(\mathbb{R}^+)$ and suppose that $T$ satisfies 
\begin{equation}\label{E:hypoonT}
\sup_{t>0}|t^{i-1}T^{(i)}(t)|\leq L<\infty, \quad i=0,1,2,...
\end{equation}
Then for every nonnegative $f \in L^p(B,\mu)$ the function $T \circ V_{r} f$ is an element of $W^{r,p}(B,\mu)$, and there is a constant $c_T>0$ depending only on $p$, $r$ and $L$ such that for every nonnegative $f \in L^p(B,\mu)$ we have
\begin{equation}\label{E:composition}
\left\|T \circ V_{r} f\right\|_{W^{r,p}(B,\mu)}\leq c_T\left\|f\right\|_{L^p}.
\end{equation}
\end{lem}

Another useful tool in our proof of Theorem \ref{T:equiv} is the following 'intermediate' description of $\cpct_{r,p}$. By $\mathcal{F}C^\infty_{b,+}$ we denote the cone of nonnegative elements of $\mathcal{F}C_b^\infty$. 

\begin{lem}\label{new}
Let $1\leq p<+\infty$ and let $r>0$ be an integer. For any open set $U\subset B$ we have 
\begin{equation}\label{E:description}
\cpct_{r,p}(U)=\inf\left\lbrace \left\|f\right\|_{L^p}^p|\: f\in \mathcal{F}C_{b,+}^\infty, \: \text{$V_rf\geq 1$ on $U$}\right\rbrace.
\end{equation}
\end{lem}
Due to Proposition \ref{P:preserve} the right hand side in (\ref{E:description}) makes sense. The lemma can be proved using standard techniques, we partially follow \cite[III. Proposition 3.5]{M}.

\begin{proof}
For $U\subset B$ open let the right hand side of (\ref{E:description}) be denoted by $\cpct_{r,p}'(U)$. Then clearly 
\begin{equation}\label{E:weakineq}
\cpct'_{r,p}(\{|V_rf|>R\})\leq R^{-p}\left\|f\right\|_{L^p}^p
\end{equation}
for all $f\in \mathcal{F}C_b^\infty$ and $R>0$. 

Now let $U\subset B$ open be fixed. The value of $\cpct_{r,p}(U)$ does not change if in its definition we require that $V_rf\geq 1+\delta$ $\mu$-a.e. on $U$ with an arbitrarily small number $\delta>0$. It does also not change if in addition we consider only nonnegative $f\in L^p$ in the definition: For any $f\in L^p$ the positivity and linearity of $V_r$ imply that $(V_r f)^+\leq V_r(f^+)$. Consequently, if $f\in L^p$ is such that $V_r f\geq 1+\delta$ $\mu$-a.e. on $U$, then also $V_r(f^+)\geq 1+\delta$ $\mu$-a.e. on $U$, and clearly $\left\|f^+\right\|_{L^p}\leq \left\|f\right\|_{L^p}$. 

Given $\varepsilon>0$ choose a nonnegative function $f\in L^p(B,\mu)$ such that $u:=V_rf\geq 1+\delta$ $\mu$-a.e. on $U$ with some $\delta>0$ and
\[\left\|f\right\|_{L^p}^p\leq \cpct_{r,p}(U)+\frac\varepsilon3.\]
Approximating $f$ by bounded nonnegative functions in $L^p(B,\mu)$, taking their cylindrical approximations, which are nonnegative as well, and smoothing by convolution in finite dimensional spaces, we can approximate $f$ in $L^p(B,\mu)$ by a sequence of nonnegative functions $(f_n)_{n=1}^\infty\subset\mathcal{F}C_{b,+}^{\infty}$, see for instance \cite[Chapter II, Theorem 5.1]{Ma97} or \cite[Theorem 7.4.5]{LMPnotes}. Clearly the functions $u_n:=V_rf_n$ satisfy $\lim_n u_n=u$ in $W^{r,p}(B,\mu)$.

By (\ref{E:weakineq}) and the convergence in $W^{r,p}(B,\mu)$ we can now choose a subsequence $(u_{n_i})_{i=1}^\infty$ such that
\[\cpct'_{r,p}(\{|u_{n_{i+1}}-u_{n_i}|>2^{-i}\})\leq 2^{-i}\quad \text{ and }\quad \left\|u_{n_{i+1}}-u_{n_i}\right\|_{L^p}\leq 2^{-2i}\]
for all $i=1,2,...$ For any $k=1,2,...$ let now 
\[A_k:=\bigcup_{i\geq k}\{|u_{n_{i+1}}-u_{n_i}|>2^{-i}\},\quad k=1,2,...\]
Then for each $k$ the sequence $(u_{n_i})_{i=1}^\infty$ is Cauchy in supremum norm on $A_k^c$. On the other hand, 
\[\mu(\{|u_{n_{i+1}}-u_{n_i}|>2^{-i}\})\leq 2^{-ip},\]
so that $\mu(A_k)\leq \sum_{i=k}^\infty 2^{-ip}$, what implies 
\[\mu\left(\bigcap_{k=1}^\infty A_k\right)=\lim_{k\to\infty} \mu(A_k)=0.\]
Consequently, setting $\overline{u}(x):=\lim_{n\to\infty} u_n(x)$ for all $x\in \bigcup_{k=1}^\infty A_k^c$ and $\overline{u}(x)=0$ for all other $x$, we obtain a $\mu$-version $\overline{u}$ of $u$. 

Now choose $l$ such that $\cpct'_{r,p}(A_l)<\frac\varepsilon3$ and then $j$ large enough so that 
\[\left\|f_{n_j}-f\right\|_{L^p}^p<\frac\varepsilon3\quad \text{and}\quad \sup_{x\in A_l^c}|u_{n_j}(x)-\overline{u}(x)|<\delta/2.\] 
Then $u_{n_j}\geq 1$ $\mu$-a.e. on some neighborhood $V$ of $U\cap A_l^c$. The topological support of $\mu$ is $B$, see for instance \cite[Theorem 3.6.1, Definition 3.6.2 and the remark following it]{B1}. Since $u_{n_j}$ is continuous by Proposition \ref{P:preserve} we therefore have $u_{n_j}\geq 1$ everywhere on $V$. Now, since $\cpct'_{r,p}$ is clearly subadditive and monotone,
\begin{align}
\cpct_{r,p}(U)\leq \cpct'_{r,p}(U)\leq \cpct'_{r,p}(V)+\cpct'_{r,p}(A_l)
&\leq \left\|f_{n_j}\right\|_{L^p}^p+\frac{\varepsilon}{3}\notag\\
&\leq\left\|f\right\|_{L^p}^p+\frac{2\varepsilon}{3}\notag\\
&\leq \cpct_{r,p}(U)+\varepsilon.\notag
\end{align}

\end{proof}

Using Lemmas \ref{lemma1} and \ref{new} we can now verify Theorem \ref{T:equiv}.

\begin{proof} 
We show (\ref{E:FCbinfty}). It suffices to consider open sets $U$. Since $\mathcal{F}C_b^\infty \subset W^{r,p}(B,\mu)$, we have
\[\cpct_{r,p}(U)\leq c_2^p\ccpct_{r,p}^{(\mathcal{F}C_b^\infty)}(U)\] 
with $c_2$ as in (\ref{E:equivnorm}), so that it suffices to show
\[\ccpct_{r,p}^{(\mathcal{F}C_b^\infty)}(U)\leq c \cpct_{r,p}(U)\]
with a suitable constant $c>0$ depending only on $r$ and $p$.

Let $T\in C^{\infty}(\mathbb{R})$ be a function such that $0\leq T\leq 1,\, T(t)=0$ for $0\leq t \leq 1/2$ and $T(t)=1$ for $t\geq 1$, and let $c_T$ be as in Lemma \ref{lemma1}. Given $\epsilon>0$, let $f\in\mathcal{F}C_{b,+}^\infty$ be such that $u:=V_rf\geq 1$ on $U$ and
\[\left\|f\right\|_{L^p}^p\leq \cpct_{r,p}(U)+\frac{\varepsilon}{c_T^p},\]
due to Lemma \ref{new} such $f$ can be found.  Clearly $T\circ u\in \mathcal{F}C_b^\infty$ and $T\circ u=1$ on $U$. Therefore, using Lemma \ref{lemma1},
we have 
\[\ccpct_{r,p}^{(\mathcal{F}C_b^\infty)}(U)\leq \left\|T\circ u\right\|_{W^{r,p}(B,\mu)}^p\leq c_T^p\left\|f\right\|_{L^p}^p\leq c_T^p\cpct_{r,p}(U)+\varepsilon,\]
and we arrive at (\ref{E:FCbinfty}) with $c_3:=1/c_T^p$ and $c_4:=c_2^p$. Since $\mathcal{F}C_b^\infty\subset W^{\infty}\subset W^{r,p}(B,\mu)$, (\ref{E:Winfty}) is an easy consequence.

\end{proof}

\section{Smooth truncations}\label{S:truncationproof}

To verify Lemma \ref{lemma1} we begin with the following generalization of \cite[formula (5.4.4) in Proposition 5.4.8]{B1}.  

\begin{prop}\label{P:CameronMartin}
Assume $p>1$ and $f\in L^p(B,\mu)$. Then for any $t>0$ and $\mu$-a.e. $x\in B$ the mapping $h\mapsto P_tf(x+h)$ from $H$ to $B$ is infinitely Fr\'echet differentiable, and given $h_1,...,h_k\in H$ we have 
\begin{multline}
\partial_{h_1}\cdots\partial_{h_k}P_t f(x)\notag\\
=\left(\dfrac{e^{-t}}{\sqrt{1-e^{-2t}}}\right)^{k}\int_{B}f(e^{-t}x+\sqrt{1-e^{-2t}}y)Q(\hat{h}_1(y),...,\hat{h}_k(y))\mu(dy),
\end{multline}
where the functions $\hat{h}_i$ are as in (\ref{E:isometry}) and $Q:\mathbb{R}^n\to \mathbb{R}$, $n\leq k$, is a polynomial of degree $k$ whose coefficients are constants or products of scalar products $\left\langle h_i, h_j \right\rangle_H$. 
If the  $h_1,...,h_k$ are elements of an orthonormal system $(g_i)_{i=1}^k$ in $H$, not necessarily distinct, then each coefficient of $Q$ depends only on the multiplicity according to which the respective element of $(g_i)_{i=1}^k$ occurs in $\left\lbrace h_1,...,h_k\right\rbrace$.
\end{prop}

\begin{proof}
The infinite differentiability was shown in \cite[Proposition 5.4.8]{B1} as a consequence of the Cameron-Martin formula. By the same arguments we can see that 
\begin{multline}
\partial_{h_1}\cdots\partial_{h_k}P_t f(x)=\int_B f(e^{-t}x+\sqrt{1-e^{-2t}}y)\times\notag\\
\times\frac{\partial^k}{\partial\lambda_1\cdots\partial\lambda_k}\varrho(t,\lambda_1h_1+...+\lambda_kh_k,y)|_{\lambda_1=...=\lambda_k=0}\:\mu(dy),
\end{multline}
where 
\[\varrho(t,h,y)=\exp\left\lbrace\frac{e^{-t}}{\sqrt{1-e^{-2t}}}\hat{h}(y)-\frac{e^{-2t}}{2(1-e^{-2t})}\left\|h\right\|_H^2\right\rbrace.\]
A straightforward calculation shows that 
\begin{multline} 
\frac{\partial^k}{\partial\lambda_1\cdots\partial\lambda_k}\varrho(t,\lambda_1h_1+...+\lambda_kh_k,y)|_{\lambda_1=...=\lambda_k=0} \notag\\
= \left(\dfrac{e^{-t}}{\sqrt{1-e^{-2t}}}\right)^{k}\:Q(\hat{h}_1(y),...,\hat{h}_k(y))
\end{multline}
with a polynomial $Q$ as stated.
\end{proof}

The next inequality is a counterpart to \cite[Proposition 3.1.8]{A}. It provides a pointwise multiplicative estimate for derivatives of potentials in terms of powers of the potential and a suitable maximal function. 

\begin{lem}\label{lemma_multest}
Let $1<q<+\infty$, let $r>0$ be an integer and let $k<r$. Then for any nonnegative Borel function $f$ on $B$ and all $x\in B$ we have 
\begin{equation}\label{E:multest}
\|D^k V_rf(x)\|_{\mathcal{H}_k}\leq c(k,q,r)\:\left(V_rf(x)\right)^{1-\frac{k}{r}}\left(\sup_{t>0} P_t(f^{q})(x)\right)^{\frac{k}{r q}}.
\end{equation}
\end{lem}

Note that lemma \ref{lemma_multest} is interesting only for $r\geq 2$. 

\begin{proof} 
Suppose $h_1,...,h_k\in H$ are members of an orthonormal system in $H$, not necessarily distinct. Then for any $\delta>0$ we have, by dominated convergence, 
\begin{align} 
D^k V_{r}f(x)&(h_1,...,h_k)\notag\\
&=\partial_{h_1}\cdots\partial_{h_k}V_rf(x)\notag\\
& = \int^{\delta}_0\int_B e^{-t}t^{r/2-1}\left(\dfrac{e^{-t}}{\sqrt{1-e^{-2t}}}\right)^k f(e^{-t}x+\sqrt{1-e^{-2t}}y)\times\notag\\
&\hspace{80pt} \times Q(\hat{h}_{1}(y),...,\hat{h}_{k}(y))\:\mu(dy)dt \notag\\ 
& +  \int^{\infty}_{\delta}\int_B e^{-t}t^{r/2-1}\left(\dfrac{e^{-t}}{\sqrt{1-e^{-2t}}}\right)^k f(e^{-t}x+\sqrt{1-e^{-2t}}y)\times\notag\\
&\hspace{80pt} \times Q(\hat{h}_{1}(y),...,\hat{h}_{k}(y))\:\mu(dy)dt\notag\\
&:=I_1(\delta)+I_2(\delta)\notag
\end{align}
with a polynomial $Q$ of degree $k$ as in Proposition \ref{P:CameronMartin}. Now let $\beta>1$ be a real number such that
\begin{equation}\label{E:parameters}
\frac{r}{2k}\leq \beta <\frac{r}{k}.
\end{equation}
H\"older's inequality yields
\begin{align}\label{E:estimateI1}
|I_1(\delta)|&\leq \left(\int^{\delta}_0\int_B e^{-t}t^{r/2-1}\left(\dfrac{e^{-t}}{\sqrt{1-e^{-2t}}}\right)^{\beta k}   
f(e^{-t}x+\sqrt{1-e^{-2t}}y)\times\right.\notag\\
&\hspace{80pt} \left.\times |Q(\hat{h}_{1}(y),...,\hat{h}_{k}(y))|^\beta \mu(dy)dt\right)^{1/\beta}\times\\
&\hspace{60pt}\left( \int^{\delta}_0\int_B e^{-t}t^{r/2-1}  f(e^{-t}x+\sqrt{1-e^{-2t}}y)\:\mu(dy)dt \right)^{1/\beta'}. \notag
\end{align}
Using the elementary inequality $e^{-t}t\leq 1-e^{-2t} $ for $t\geq 0$ and (\ref{E:parameters}), 
\[ e^{-t}t^{r/2-1}\left(\dfrac{e^{-t}}{\sqrt{1-e^{-2t}}}\right)^{\beta k}\leq (1-e^{-2t})^{r/2-k\beta/2-1}\:e^{-2t},\]
so that another application of H\"older's inequality, now with $q$, shows that the first factor on the 
right hand side of (\ref{E:estimateI1}) is bounded by 
\begin{align}
&\left(\int^{\delta}_0\int_B  f(e^{-t}x+\sqrt{1-e^{-2t}}y)^q \mu(dy)(1-e^{-2t})^{r/2-k\beta/2-1}e^{-2t}dt\right)^{1/(\beta q)}\times\notag\\
&\times \left(\int_B |Q(\hat{h}_{1}(y),...,\hat{h}_{k}(y))|^{\beta q'}\mu(dy) \int^{\delta}_0 (1-e^{-2t})^{r/2-k\beta/2-1}e^{-2t}dt\right)^{1/(\beta q')}.\notag
\end{align}
According to Proposition \ref{P:CameronMartin} the coefficients of the polynomial $Q$ are bounded for fixed $k$, and since its degree does not exceed $k$, it involves only finitely many distinct products of powers of the functions $\hat{h}_i$. Together with the fact that each $\hat{h}_i$ is $N(0,1)$-distributed, this implies that there is a constant $c_1(k,q,\beta)>0$, depending on $k$ but not on the particular choice of the elements $h_1,...,h_k$, such that
\[\left(\int_B |Q(\hat{h}_{1}(y),...,\hat{h}_{k}(y))|^{\beta q'}\mu(dy)\right)^{1/(\beta q')}< c_1(k,q,\beta).\]
Taking into account (\ref{E:parameters}), we therefore obtain
\begin{multline}\label{E:I1final}
|I_1(\delta)| \leq c_1(k,q,\beta)\:\left(\frac{r}{2}-\frac{\beta k}{2}\right)^{-1/\beta}(1-e^{-2\delta})^{r/(2\beta)-k/2} \times\\
\hspace{40pt}\times\left(V_rf(x)\right)^{1/\beta'}\left(\sup_{t>0} P_t(f^{q})(x)\right)^{\frac{1}{\beta q}}.
\end{multline}
To estimate $I_2(\delta)$ let 
\begin{equation}\label{E:parameters2}
\frac{r}{k}<\gamma. 
\end{equation}
In a similar fashion we can then obtain the estimate
\begin{multline}\label{E:I2final}
|I_2(\delta)| \leq c_2(k,q,\gamma)\:\left(\frac{r}{2}-\frac{\gamma k}{2}\right)^{-1/\gamma}(1-e^{-2\delta})^{r/(2\gamma)-k/2} \times\\
\hspace{40pt}\times\left(V_rf(x)\right)^{1/\gamma'}\left(\sup_{t>0} P_t(f^{q})(x)\right)^{\frac{1}{\gamma q}},
\end{multline}
where $c_2(k,q,\gamma)>0$ is a constant depending on $n$ but not on the particular choice of $h_1,...,h_k$. 

We finally choose suitable $\delta>0$. The function 
\[\delta\mapsto (1-e^{-2\delta}),\quad \delta>0,\] 
can attain any value in $(0,1)$. Since Jensen's inequality implies
\begin{equation}
(V_rf(x))^q\leq \sup_{t>0}(P_t(f)(x))^q\leq \sup_{t>0} P_t(f^q)(x),
\end{equation}
we have $\sup_{t>0}(P_t(f^q)(x))^{1/q}\geq V_rf(x)$ and can choose $\delta>0$ such that 
\begin{equation}\label{E:desired}
(1-e^{-2\delta})= \frac{V_rf(x)^{2/r}}{2\sup_{t>0}(P_t(f^q)(x))^{2/(qr)}},
\end{equation}
note that the denominator cannot be zero unless $f$ is zero $\mu$-a.e.
Combining with (\ref{E:I1final}) and (\ref{E:I2final}) we obtain 
\begin{multline}
|D^k V_{r}f(x)(h_1,...,h_k)|\notag\\
\leq \left\lbrace c'_1(k,q,\beta)\:\left(\frac{r}{2}-\frac{\beta k}{2}\right)^{-1/\beta} + c'_2(k,q,\gamma)\:\left(\frac{r}{2}-\frac{\gamma k}{2}\right)^{-1/\gamma}\right\rbrace\times\notag\\
\times \left(V_rf(x)\right)^{1-k/r}\left(\sup_{t>0} P_t(f^{q})(x)\right)^{k/(qr)}
\end{multline}
for some constants $c'_1(k,q,\beta),\ c'_2(k,q,\gamma)$. For any given $r$ there exist only finitely many numbers $k<r$ and for any such $k$ numbers $\beta$ and $\gamma$
as in (\ref{E:parameters}) and (\ref{E:parameters2}) can be fixed. Using Proposition \ref{P:HSnormestimate} we can therefore find a constant $c(k,q,r)$ depending only on $k$, $q$ and $r$ such that (\ref{E:multest}) holds.
\end{proof}

We prove Lemma \ref{lemma1}, basically following the method of proof used for \cite[Theorem 3.3.3]{A}.

\begin{proof}
If $r=1$ then $T$ has a bounded first derivative, and the desired bound is immediate from the definition of the norm $\left\|\cdot\right\|_{W^{1,p}}$, the chain rule for the gradient $D$ and Meyer's equivalence, \cite[Theorem 4.4]{S}. In the following we therefore assume $r\geq 2$. 

We verify that for any $k\leq r$ the inequality
\begin{equation}\label{E:target}
\left\|D^k (T\circ V_r f)\right\|_{L^p(B,\mu, \mathcal{H}^k)}\leq c(k,L,p,r)\:\left\|f\right\|_{L^p(B,\mu)}
\end{equation}
holds with a constant $c(k,L,p,r)>0$ depending only on $k$, $L$, $p$ and $r$.  If so, then summing up yields
\[\left\|T\circ V_rf\right\|_{W^{r,p}(B,\mu)}=\sum_{k=0}^r \left\|D^k(T\circ V_r f)\right\|_{L^p(B,\mu, \mathcal{H}^r)}\leq c_T\:\left\|f\right\|_{L^p(B,\mu)}\]
with a constant $c_T>0$ depending on $L$, $p$ and $r$, as desired.

To see (\ref{E:target}) suppose $k\leq r$ and that $h_1,...,h_k$ are members of an orthonormal system $(g_i)_{i=1}^k$, not necessarily distinct. To simplify notation, we use multiindices with respect to this orthonormal system: Given a multiindex $\alpha=(\alpha_1,...,\alpha_k)$ we write $D^\alpha:=\partial_{g_1}^{\alpha_1}\cdots\partial_{g_k}^{\alpha_k}$,
where for $\beta=0,1,2,...$, a function $u:B\to\mathbb{R}$ and an element $g\in H$ we define $\partial_g^\beta u$ as the image of $u$ under the application $\beta$ differentiations in direction $g$, 
\[\partial^\beta_g u(x):=\partial_g\cdots \partial_g u(x)=D^\beta u(x)(g,...,g).\]
Now let $\alpha$ be a multiindex such that $D^\alpha=\partial_{h_1}\cdots \partial_{h_k}$. Then clearly $|\alpha|=k$. Moreover, we have
\begin{multline}
D^{\alpha}(T\circ V_{r}f)(x)\notag\\
=\sum^{k}_{j=1}T^{(j)}\circ{V_r f(x)}\sum C_{\alpha^1,...,\alpha^j}D^{\alpha^1}{V_{r}f(x)}\cdots D^{\alpha^j}{V_{r}f(x)}
\end{multline}
by the chain rule, where the interior sum is over all  $j$-tuples $(\alpha^1,...,\alpha^j)$ of multiindices $\alpha^i$ such that $|\alpha^i|\geq 1$ for all $i$ and $\alpha^1+\alpha^2+...+\alpha^j=\alpha$. 
The interior sum has $\binom{k-1}{j-1}$ summands. The $C_{\alpha^1,...,\alpha^j}$ are real valued coefficients, and since there are only finitely many different $C_{\alpha^1,...,\alpha^j}$, there exists a constant $C(k)>0$ which for all multiindices $\alpha$ with $|\alpha |= k$  dominates these constants, $C_{\alpha^1,...,\alpha^j}\leq C(k)$. 
In particular, $C(k)$ does not depend on the particular choice of the elements $h_1,...,h_k$. More explicit computations can for instance be obtained using \cite{HEMM03}.

The hypothesis (\ref{E:hypoonT}) on $T$ implies 
\[|D^{\alpha}(T\circ V_{r}f)(x)|\leq c(k)L\:\sum^{k}_{j=1} (V_rf(x))^{1-j}\sum |D^{\alpha^1}{V_{r}f(x)}\cdots D^{\alpha^j}{V_{r}f(x)}|\]
with a constant $c(k)>0$ depending only on $k$ and with $L$ being as in (\ref{E:hypoonT}). Since $\sum_{i=1}^j(1-|\alpha^i|/k)=j-|\alpha|/k=j-1$ and 
\[|D^{\alpha^i}{V_{r}f(x)}|\leq \|D^{|\alpha^i|}{V_{r}f(x)}\|_{\mathcal{H}_{|\alpha^i|}},\] 
Lemma \ref{lemma1} implies that 
\begin{multline}
\sum^{k}_{j=2} (V_rf(x))^{1-j}\sum |D^{\alpha^1}{V_{r}f(x)}\cdots D^{\alpha^j}{V_{r}f(x)}|\notag\\
\leq \sum^{k}_{j=2} (V_rf(x))^{1-j}\sum \|D^{|\alpha^1|}{V_{r}f(x)}\|_{\mathcal{H}_{|\alpha^1|}}\cdots \|D^{|\alpha^j|}{V_{r}f(x)}\|_{\mathcal{H}_{|\alpha^j|}} \\
\leq c(k,q,r)\sum_{j=2}^k\binom{k-1}{j-1}  (\sup_{t>0} P_t(f^q)(x))^{1/q},
\end{multline}
where $1<q<+\infty$ is arbitrary and $c(k,q,r)>0$ is a constant depending only on $k$, $q$ and $r$. For the case $j=1$ we have 
\[|D^{\alpha}{V_{r}f(x)}|\leq \|D^{k}{V_{r}f(x)}\|_{\mathcal{H}_{k}}.\] 
Taking the supremum over all $h_1,...,h_k\in H$ as above we obtain
\[\left\|D^k T\circ V_rf(x)\right\|_{\mathcal{H}^k}\leq c(k,L,q,r)\left[(\sup_{t>0} P_t(f^q)(x))^{1/q}+\|D^{k}{V_{r}f(x)}\|_{\mathcal{H}_{k}}\right]\]
with a constant $c(k,L,q,r)>0$ by Proposition \ref{P:HSnormestimate}.

Fixing $1<q<p$ and using the boundedness of the semigroup maximal function, \cite[Theorem 3.3]{S}, we see that there is a constant 
$c(p,q)>0$ depending only on $p$ and $q$ such that
\[\big\|(\sup_{t>0} P_t(f^q))^{1/q}\big\|_{L^{p}(B,\mu)}\leq c(p,q)\:\left\|f\right\|_{L^p(B,\mu)}.\]
On the other hand, by (\ref{E:equivnorm}), we have  
\[\left\|D^k V_r f\right\|_{L^p(B,\mu, \mathcal{H}_k)}\leq \frac{1}{c_1}\:\left\|f\right\|_{L^p(B,\mu)}.\]
Combining, we arrive at (\ref{E:target}).

\end{proof}

\section{$L^p$-uniqueness}\label{S:Lp}

We discuss related uniqueness problems for the Ornstein Uhlenbeck operator $\mathcal{L}$ and its integer powers. 

Recall first that a densely defined operator $(L,\mathcal{A})$ on $L^p(B,\mu)$, $1\leq p<+\infty$ is said to be \emph{$L^p$-unique} if there is only one $C_0$-semigroup on $L^p(B,\mu)$ whose generator extends $(L,\mathcal{A})$, see e.g. \cite[Chapter I b), Definition 1.3]{E}. If $(L,\mathcal{A})$ has an extension generating a $C_0$-semigroup on $L^p(B,\mu)$ then $(L,\mathcal{A})$ is $L^p$-unique if and only if the closure of $(L,\mathcal{A})$ generates a $C_0$-semigroup on $L^p(B,\mu)$, see \cite[Chapter I, Theorem 1.2 of Appendix A]{E}. 

From (\ref{E:equivnorm}) it follows that for any $m=1,2,...$  and $1<p<+\infty$ we have $\mathcal{D}((-\mathcal{L})^{m})=W^{2m,p}(B,\mu)$. The density of $\mathcal{F}C_b^\infty$ and $W^\infty$ in the spaces $W^{2m,p}(B,\mu)$ and the completeness of the latter imply that $((-\mathcal{L})^{m}, W^{2m,p}(B,\mu))$ is the closure in $L^p(B,\mu)$ of $((-\mathcal{L})^{m},\mathcal{F}C_b^\infty)$ and also of $((-\mathcal{L})^{m},W^\infty)$. 

Since obviously $(P_t)_{t>0}$ is a $C_0$-semigroup, $(\mathcal{L},\mathcal{F}C_ b^\infty)$ and $(\mathcal{L},W^\infty)$ are $L^p$-unique in all $L^p(B,\mu)$, $1\leq p<+\infty$. To discuss the its powers $-(-\mathcal{L})^m$ for $m\geq 2$ we quote well known facts to provide a sufficient condition for them to generate $C_0$-semigroups. Since $(P_t)_{t>0}$ is a symmetric Markov semigroup on $L^2(B,\mu)$, for any $1<p<+\infty$ the operator $\mathcal{L}=\mathcal{L}^{(p)}$ generates a bounded holomorphic semigroups on $L^p(B,\mu)$ with angle $\theta$ satisfying $\frac{\pi}{2}-\theta\leq \frac{\pi}{2}|\frac{2}{p}-1|$,
see for instance \cite[Theorem 1.4.2]{D89}. On the other hand \cite[Theorem 4]{DeL87} tells that if $L$ is the generator of a bounded holomorphic semigroup with angle $\theta$ satisfying $\frac{\pi}{2}-\theta<\frac{\pi}{2m}$, then also $-(-L)^m$ generates a bounded holomorphic semigroup. Combining, we can conclude that $-(-\mathcal{L})^m$ generates a bounded holomorphic semigroup on $L^p(B,\mu)$ and therefore in particular a (bounded) $C_0$-semigroup if
\begin{equation}\label{E:condition}
|\frac{2}{p}-1|<\frac{1}{m}\ .
\end{equation}
\cite[Theorem 8]{DeL88} shows that (up to a discussion of limit cases) this is a sharp condition for $-(-\mathcal{L})^m$ to generate a bounded $C_0$-semigroup. For $1<p<+\infty$ this also recovers the $L^p$-uniqueness in the case $m=1$. For $p=2$ condition (\ref{E:condition}) is always satisfied. Alternatively we can conclude the generation of $C_0$-semigroups on $L^2(B,\mu)$ directly from the spectral theorem. 

For later use we fix the following fact.

\begin{prop}
Let $1<p<+\infty$ and let $m>0$ be an integer satisfying (\ref{E:condition}). Then the operators $(-(-\mathcal{L})^m,\mathcal{F}C_b^\infty)$ and $(-(-\mathcal{L})^m, W^\infty)$ are $L^p$-unique in $L^p(B,\mu)$. In particular, they are essentially self-adjoint in $L^2(B,\mu)$ for all $m>0$.
\end{prop}

The last statement is true because a semi-bounded symmetric operator $(L,A)$ on $L^2(B,\mu)$ is $L^2$-unique if and only if it is essential self-adjoint, see \cite[Chapter I c), Corollary 1.2]{E}.

Here we are interested in $L^p$-uniqueness after the removal of a small closed set $\Sigma\subset B$ of zero measure. This is similar to our discussion in \cite{HiKaMa} and, in a sense, similar to a removable singularities problem, see for instance \cite{Mazja72} or \cite{Mazja} or \cite[Section 2.7]{A}.

Let $\Sigma\subset B$ be a closed set of zero Gaussian measure and $N:=B\setminus \Sigma$.
We define 
\[\mathcal{F}C^{\infty}_b(N):=\{f\in \mathcal{F}C^{\infty}_b|\: \text{$f=0$ on an open neighborhood of $\Sigma$} \}\]
and 
\[W^\infty(N):=\{f\in W^\infty|\: \text{$\widetilde{f}=0$ q.s. on an open neighborhood of $\Sigma$} \}.\]
The $L^p$-uniqueness of $-(-\mathcal{L})^m$, restricted to $\mathcal{F}C^{\infty}_b(N)$ and $W^\infty(N)$, respectively, now depends on the size of the set $\Sigma$. If it is small enough not to cause additional boundary effects then from the point of view of operator extensions it is removable.

\begin{theo}\label{T:Lp}
Let $1<p<+\infty$, let $m>0$ be an integer and assume that $\Sigma\subset B$ is a closed set of zero measure $\mu$. Write $N:=B\setminus \Sigma$.
\begin{enumerate}
\item[(i)] If $\cpct_{2m,p}(\Sigma)=0$ then the closure of $(-(-\mathcal{L})^m,\mathcal{F}C_b^\infty(N))$ in $L^p(B,\mu)$ is \[(-(-\mathcal{L})^m,W^{2m,p}(B,\mu)).\] 
If in addition $m$ satisfies (\ref{E:condition}) then $(-(-\mathcal{L})^m,\mathcal{F}C_b^\infty(N))$ is $L^p$-unique.
\item[(ii)] If $(-(-\mathcal{L})^m,\mathcal{F}C_b^\infty(N))$ is $L^p$-unique, then $\cpct_{2m,p}(\Sigma)=0$.
\end{enumerate}
The same statements are true with $W^\infty(N)$ in place of $\mathcal{F}C_b^\infty(N)$.
\end{theo}

\begin{proof}
To see (i) suppose that $\cpct_{2m,p}(\Sigma)=0$. Let $((-\mathcal{L})^m,\mathcal{D}((-\mathcal{L})^m))$ denote the closure of $((-\mathcal{L})^m, \mathcal{F}C_b^\infty(N))$ in $L^p(B,\mu)$. Since $\mathcal{F}C_b^\infty(N)\subset \mathcal{F}C_b^\infty$ we trivially have 
\[\mathcal{D}((-\mathcal{L})^m)\subset W^{2m,p}(B,\mu),\] 
and it remains to show the converse inclusion. 

Given $u\in W^{2m,p}(B,\mu)$, let $(u_j)_{j=1}^\infty\subset \mathcal{F}C^{\infty}_b$ be a sequence approximating $u$ in $W^{2m,p}(B,\mu)$. By Theorem \ref{T:equiv} there is a sequence $(v_l)_{l=1}^\infty\subset \mathcal{F}C^{\infty}_b$ such that $\lim_{l\to\infty} v_l=0$ in $W^{2m,p}(B,\mu)$ and for each $l$ the function $v_l$ equals one on an open neighborhood of $\Sigma$. Set $w_{jl}:=(1-v_l)u_j$ to obtain functions $w_{jl}\in \mathcal{F}C^{\infty}_b(N).$ Now let $j$ be fixed. For any $1\leq k\leq 2m$ let $h_1,...,h_k$ be members of an orthonormal system $(g_i)_{i=1}^k$, not necessarily distinct. As in the proof of Lemma \ref{lemma1} we use multiindex notation with respect to this orthonormal system. Let $\alpha$ be such that $D^\alpha=\partial_{h_1}\cdots\partial_{h_k}$. Then, by the general Leibniz rule,
\[D^\alpha (u_j-w_{jl})(x)=D^\alpha (u_jv_l)(x)=\sum_{\beta\leq \alpha}\binom{\alpha}{\beta}D^\beta u_j(x)D^{\alpha-\beta}v_l(x),\]
where for two multiindices $\alpha$ and $\beta$ we write $\beta\leq \alpha$ if $\beta_i\leq \alpha_i$ for all $i=1,...,k$. For any such $\beta$ we clearly have 
\[|D^\beta u_j(x)|\leq \left\|D^{|\beta|}u_j(x)\right\|_{\mathcal{H}_{|\beta|}}\ \text{and }\ |D^{\alpha-\beta}v_l(x)|\leq \left\|D^{|\alpha-\beta|}v_l(x)\right\|_{\mathcal{H}_{|\alpha-\beta|}},\]
and taking the supremum over all $h_1,...,h_k$ as above,
\[\left\|D^k(u_j-w_{jl})(x)\right\|_{\mathcal{H}_k}\leq c(k)\max_{n\leq k}\left\|D^n u_j(x)\right\|_{\mathcal{H}_n}\max_{n\leq k}\left\|D^n v_l(x)\right\|_{\mathcal{H}_n}\]
with a constant $c(k)>0$ depending only on $k$. Taking into account that 
\[\sup_{x\in B} \left\|D^n u_j(x)\right\|_{\mathcal{H}_n}<+\infty \]
for any $n\geq 1$ and summing up, we see that 
\begin{align}
\lim_l \sum_{k=1}^{2m}& \left\|D^k(u_j-w_l)\right\|_{L^p(B,\mu,\mathcal{H}_k)}\notag\\
&\leq c(m)\max_{n\leq 2m}\sup_{x\in B} \left\|D^n u_j(x)\right\|_{\mathcal{H}_n}\lim_l \left\|v_l\right\|_{W^{2m,p}}\notag\\
&=0,\notag
\end{align}
here $c(m)>0$ is a constant depending on $m$ only. Since $u_j$ is bounded, we also have $\lim_l (u_j-w_{jl})=\lim_l u_jv_l= 0$ in $L^p(B,\mu)$ so that 
\[\lim_l w_{jl}=u_j\quad \text{in $W^{2m,p}(B,\mu)$,}\] 
what implies $u\in \mathcal{D}((-\mathcal{L})^m)$ and therefore 
\[W^{2m,p}(B,\mu)\subset \mathcal{D}((-\mathcal{L})^m).\] 

To see (ii) suppose that $(-(-\mathcal{L})^m,\mathcal{F}C^{\infty}_b(N))$ is $L^p$-unique in $L^p(B,\mu)$. Then its unique extension must be $(-(-\mathcal{L})^m, W^{2m,p}(B,\mu))$. Let $u\in \mathcal{F}C_b^{\infty}$ be a function that equals one on a neighborhood of $\Sigma$. Since $\mathcal{F}C_b^{\infty}\subset W^{2m,p}(B,\mu)$ and by hypothesis $\mathcal{F}C_b^{\infty}(N)$ is dense in $W^{2m,p}(B,\mu)$, we can find a sequence $(u_l)_l\subset \mathcal{F}C_b^{\infty}(N)$ approximating $u$ in $W^{2m,p}(B,\mu)$. The functions $e_l:=u-u_l$ then are in $\mathcal{F}C_b^{\infty}$, each equals one on an open neighborhood of $\Sigma$, and they converge to zero in $W^{2m,p}(B,\mu)$, so that by Theorem \ref{T:equiv} we have
\[\cpct_{2m,p}(\Sigma)\leq c_2 \lim_l\left\|e_l\right\|_{W^{2m,p}}=0.\]

The proof for $W^\infty$ is similar.
\end{proof}

\section{Comments on Gaussian Hausdorff measures}\label{S:Hausdorff}

For finite dimensional Euclidean spaces the link between Sobolev type capacities and Hausdorff measures is well known and the critical size of a set $\Sigma$ in order to have $(r,p)$-capacity zero or not is, roughly speaking,  determined by its Hausdorff codimension, see e.g. \cite[Chapter 5]{A}. For Wiener spaces one can at least provide a partial result of this type.

Hausdorff measures on Wiener spaces of integer codimension had been introduced in \cite[Section 1]{FdP92}. We briefly sketch their method but allow non-integer codimensions, this is an effortless generalization and immediate from their arguments.

Given an $m$-dimensional Euclidean space $F$ and a real number $0\leq d\leq m$ the spherical Hausdorff measure $\mathcal{S}^d$ of dimension $d$ can be defined as follows: For any $\varepsilon >0$ set 
\begin{multline}
\mathcal{S}^d_\varepsilon(A):=\inf\left\lbrace \sum_{i=1}^\infty r_i^d: \left\lbrace B_i\right\rbrace_{i=1}^\infty\ \text{is a collection of balls }\right.\notag\\
\left. \quad \text{of radius $r_i<\varepsilon/2$ such that $A\subset \bigcup_{i=1}^\infty B_i$}\right\rbrace,
\end{multline}
and finally, $\mathcal{S}^d(A):=\sup_{\varepsilon >0} \mathcal{S}^d_\varepsilon(A)$, $A\subset F$. A priori $\mathcal{S}^d$ is an outer measure, but its $\sigma$-algebra of measurable sets contains all Borel sets. For any $0\leq d\leq m$ and we define 
\[\theta_d^F(A):=(2\pi)^{-m/2}\int_A \exp\left(-\frac{|y|_F^2}{2}\right)\mathcal{S}^{m-d}(dy),\]
for Borel sets $A\subset F$, \cite[1. Definition]{FdP92}, by approximation from outside it extends to an outer measure on $F$, defined in particular for any analytic set. Recall that a set $A\subset F$ is called \emph{analytic} if it is a continuous image of a Polish space.

We return to the abstract Wiener space $(B,\mu,H)$. Let $d\geq 0$ be a real number and let $F$ be a subspace of $H$ of finite dimension $m\geq d$. Let $p^F$ denote the orthogonal projection from $H$ onto $F$, it extends to a linear projection $p^F$ from $B$ onto $F$ which is $(r,p)$-quasi continuous for all $r$ and $p$, \cite[11. Th\'eor\`{e}me]{FdP91}. We write $\widetilde{F}$ for the kernel of $p^F$. The spaces $B$ and $F\times \widetilde{F}$ are isomorphic under 
the map $p^F\times (I-p^F)$. If $A\subset B$ is analytical and for any $x\in\widetilde{F}$ the section with respect to the above product is denotes by $A_x\subset F$, then for any $a\in\mathbb{R}$ the set $\{ x\in\widetilde{F}: \theta_d^F(A_x)>a\}$  is analytic up to a slim set, as shown in \cite[4. Lemma]{FdP92}. We follow \cite[5. Definition]{FdP92} and set $\mu^F(B):=\mu((I-p^F)^{-1}(B))$ for any analytic subset $B$ of $F$. Then by \cite[4. Lemma]{FdP92} we can define 
\[\varrho_d^F(A):=\int_B \theta^F_d(A_x)\mu(dx)\]
for any analytic subset $A$ of $B$. As in \cite[8. Definition]{FdP92} we define the \emph{Gaussian Hausdorff measure $\varrho_d$ of codimension $d\geq 0$} by 
\[\varrho_d(A):=\sup\left\lbrace \varrho_d^F(A): F\subset H\ \text{ and $d\leq \dim F<+\infty$}\right\rbrace\]
for any analytic set $A\subset B$. Restricted to the Borel $\sigma$-algebra it is a Borel measure. The next result follows in the same way as \cite[9. Theorem]{FdP92} from \cite[32. Th\'eor\`{e}me]{FdP91} and \cite{MazjaKhavin}, see also \cite[Theorem 5.1.13]{A}.

\begin{theo}
If a Borel set $A\subset B$ satisfies $\cpct_{r,p}(A)=0$, then $\varrho_d(A)=0$ for all $d<rp$.
\end{theo}

Combined with Theorem \ref{T:Lp} this yields a necessary codimension condition which is similar as in the case of Laplacians on Euclidean spaces, \cite{AGMST13, HiKaMa}.

\begin{cor}\label{C:codimension}
Assume $1<p<+\infty$. Let $\Sigma\subset B$ be a closed set of zero measure and $N:=B\setminus \Sigma$. 

If $(-(-\mathcal{L})^m, \mathcal{F}C_b^\infty(N))$ is $L^p$-unique, then 
\[\varrho_d(\Sigma)=0\quad \text{for all $d<2mp$.}\] 
In particular, if $(\mathcal{L}, \mathcal{F}C_b^\infty(N))$ is essentially self-adjoint, then 
\[\varrho_d(\Sigma)=0\quad \text{ for all $d<4$.}\]

The same is true with $W^\infty(N)$ in place of $\mathcal{F}C_b^\infty(N)$.
\end{cor}

\section[Comments on stochastic processes]{Comments on stochastic processes}\label{S:processes}
 
We finally like to briefly point out connections to known Kakutani type theorems for related multiparameter Ornstein-Uhlenbeck processes. The connection between Gaussian capacities, \cite{FdP91}, and the hitting behavious of multiparameter processes, \cite{Hi93, HiSo95, HiSo95Pota}, has for instance been investigated in \cite{Ba94, So93, So93b}. We briefly sketch the construction and main result of \cite{So93b}, later generalized in \cite{Ba94}. 

Let $\Theta^{(0)}:=B$ and for integer $k\geq 1$, $\Theta^{(k+1)}(B):=C(\mathbb{R}_+,\Theta^{(k)}(B))$. The space $\Theta^k(B)$ can be identified with $C(\mathbb{R}_+^k,B)$. Moreover, set $\mu^{(0)}:=\mu$, $T_t^{(0)}:=P_t$, $t>0$, and let $Z^{(1)}$ be the Ornstein-Uhlenbeck process taking values in $\Theta^{(0)}(B)=B$ with semigroup $T_t^{(0)}$ and initial law $\mu^{(0)}$. Let $\mu^{(1)}$ denote the law of the process $Z^{(1)}$, clearly a centered Gaussian measure on $\Theta^{(1)}(B)$. Next, let $(T_t^{(1)})_{t<0}$ be the Ornstein-Uhlenbeck semigroup on $\Theta^{(1)}(B)$ defined by
\[T_t^{(1)}f(x)=\int_{\Theta^{(1)}(B)} f(e^{-t}x+\sqrt{1-e^{-2t}}y)\mu^{(1)}(dy),\quad x\in \Theta^{(1)}(B),\]
for any bounded Borel function $f$ on $\Theta^{(1)}(B)$, and let $Z^{(2)}$ be the Ornstein-Uhlenbeck process taking values in $\Theta^{(1)}(B)$ with semigroup $(\Theta^{(1)})_{t>0}$ and initial law $\mu^{(1)}$. Iterating this construction yields, for any integer $r\geq 1$, an Ornstein-Uhlenbeck process $Z^{(r)}$ taking values in $\Theta^{(r-1)}(B)$. This process may also be viewed as an $r$-parameter process $Z^{(r)}=(Z^{(r)}_{\mathbf{t}})_{\mathbf{t}\in\mathbb{R}^r_+}$ taking values in $B$. Now \cite[\S 6, Th\'eor\`{e}me 1]{So93b} tells that a Borel set $A\subset B$ has zero $(r,2)$-capacity $\cpct_{r,2}(A)=0$ if and only if the event 
\[\left\lbrace \text{there exists some $\mathbf{t}\in\mathbb{R}_+^r$ such that $Z^{(r)}_{\mathbf{t}}\in A$}\right\rbrace\]
has probability zero. See also \cite[13. Corollary]{Ba94}.

Combined with Theorem \ref{T:Lp} this result gives a preliminary characterization of $L^2$-uniqueness (that is, essential self-adjointness) in terms of the hitting behaviour of the $2m$-parameter Ornstein-Uhlenbeck process $(X_{\mathbf{t}}^{(m)})_{\mathbf{t}\in \mathbb{R}_+^{2m}}$. 

\begin{cor} Let $m>0$ be an integer. Let $\Sigma\subset B$ be a closed set of zero measure and $N:=B\setminus \Sigma$. 
The operators $(-(-\mathcal{L})^m,\mathcal{F}C_b^\infty(N))$ and $(-(-\mathcal{L})^m,W^\infty(N))$ are $L^2$-unique (resp. essentially self-adjont) if and only if $Z^{(2m)}$ does not hit $\Sigma$ with positive probability.
\end{cor}

A more causal connection between uniqueness problems for operators and classical probability should involve certain branching diffusions rather than multiparameter processes, but even for finite dimensional Euclidean spaces the problem is not fully settled and remains a future project.


\begin{thebibliography}{XXX}

\bibitem{A76}
D.R. Adams, \emph{On the existence of capacitary strong type
estimates in $\mathbb{R}^n$}, Ark. Math. {\bf 14} (1-2) (1976), 125--140.

\bibitem{A} D.R. Adams, L.I. Hedberg, \emph{Function Spaces and Potential Theory}, Springer, New York, 1996.

\bibitem{AP} D.R. Adams, J. Polking, \emph{The equivalence of two definitions of capacity}, Proc. Amer. Math. Soc. {\bf 37}(2) (1973), 529--534.

\bibitem{AR} S. Albeverio, M. R\"ockner, \emph{Classical Dirichlet forms on topological vector spaces - closability and a Cameron-Martin formula}, J. Funct. Anal. {\bf 88}(2) (1990), 395--436.

\bibitem{AGMST13}
M.S. Ashbaugh, F. Gesztesy, M. Mitrea, R. Shterenberg, G. Teschl, \emph{A Survey on the Krein-von Neumann Extension, the Corresponding Abstract Buckling Problem, and Weyl-type Spectral Asymptotics for Perturbed Krein Laplacians in Nonsmooth Domains}, In: Mathematical Physics, Spectral Theory and Stochastic Analysis, ed. by M. Demuth and W. Kirsch, Operator Theory: Adv. and Appl. 232, 2013, pp. 1--106.

\bibitem{Ba94}
J. Bauer, \emph{Multiparameter processes associated with Ornstein-Uhlenbeck semigroups}, In: Classical and modern potential theory and applications, vol. {\bf 430}, NATO ASI Series, 1994, pp. 41--56.

\bibitem{B} 
V.I. Bogachev, M. R\"ockner, \emph{Mehler formula and capacities for infinite dimensional Ornstein-Uhlenbeck processes with general linear drift}, Osaka J. Math. {\bf 32}(2) (1995), 237--274.

\bibitem{B1} 
V.I. Bogachev, \emph{Gaussian Measures}, Math. Surveys and Monographs, Volume 62, Amer. Math. Soc., Providence, 1998.

\bibitem{B1a} 
V.I. Bogachev, \emph{Differentiable Measures and the Malliavin Calculus}, Math. Surveys and Mnographs, Volume 164, Amer. Math. Soc., Providence, 2010.  

\bibitem{B2} 
V.I. Bogachev, A. Pilipenko, A. Shaposhnikov, \emph{Sobolev functions on infinite-dimensional domains}, J.  Math. Anal. Appl. {\bf 419}(2) (2014), 1023--1044.

\bibitem{Bo} 
N. Bouleau, F. Hirsch, \emph{Dirichlet Forms and Analysis on Wiener Space}, Walter de Gruyter, Berlin, 1991.

\bibitem{CdV82}
Y. Colin de Verdi\`ere, \emph{Pseudo-Laplaciens. I}, Ann. Inst. Fourier {\bf 32} (1982), 275--286.

\bibitem{DPL}  
G. DaPrato, A. Lunardi, \emph{Maximal $L^2$ regularity for Dirichlet problems in Hilbert spaces}, J. Math. Pures et Appl. {\bf 99}(6) (2013), 741--765.

\bibitem{DPL2} 
G. DaPrato, A. Lunardi, \emph{On the Dirichlet semigroup for Ornstein-Uhlenbeck operators in subsets of Hilbert spaces}, J. Funct. Anal. {\bf 259} (2010), 2642--2672.


\bibitem{D89}
E.B. Davies, \emph{Heat Kernels and Spectral Theory},
Cambridge Univ. Press, Cambridge, 1989.

\bibitem{DeL87}
R. deLaubenfels, \emph{Powers of generators of holomorphic semigroups}, 
Proc. Amer. Math. Soc. {\bf 99}(1) (1987), 105--108.

\bibitem{DeL88}
R. deLaubenfels, \emph{Totally accretive operators}, 
Proc. Amer. Math. Soc. {\bf 103}(2) (1988), 551--556.

\bibitem{E} 
A. Eberle, \emph{Uniqueness and Non-Uniqueness of Semigroups Generated by Singular Diffusion Operators}, Lect. Notes Math. vol. {\bf 1718}, Springer, Berlin, 1999.

\bibitem{HEMM03}
L. Hern\'andez Encinas, J. Mu\~{n}oz Masqu\'e,  \emph{A short proof of the generalized Fa\`a di Bruno's formula}, Appl. Math. Letters {\bf 16}(6) (2003), 975--979.

\bibitem{FdP91}
D. Feyel, A. de La Pradelle, \emph{Capacit\'es gaussiennes}, Ann. Inst. Fourier {\bf 41}(1) (1991), 49--76.

\bibitem{FdP92}
D. Feyel, A. de La Pradelle, \emph{Hausdorff measures on the Wiener space}, Pot. Anal. {\bf 1} (1992), 177--189.

\bibitem{Gross67}
L. Gross, \emph{Abstract Wiener spaces}, Proc. of the Fifth Berkeley Symposium on Math. Stat. and Probab., Vol. 2: Contributions to Probability Theory, Part 1, Univ. of California Press, Berkeley, Calif., 1967, pp. 31--42.

\bibitem{H} 
M. Hino, \emph{On Dirichlet spaces over convex sets in infinite dimensions}, Finite and infinite dimensional analysis in honor of Leonard Gross (New Orleans, LA, 2001), Contemp. Math. {\bf 317}, Amer. Math. Soc., Providence, 2003, pp. 143--156.

\bibitem{H1} 
M. Hino, \emph{Dirichlet spaces on $H-$convex sets in Wiener space}, Bull. Sci. Math. {\bf 135}(6-7) (2011), 667--683.

\bibitem{HiKaMa} 
M. Hinz, S. Kang, J.  Masamune, \emph{Probabilistic characterizations of essential self-adjointness and removability of singularities}, Sci. Journal of Volgograd State Univ. Math. Physics and Comp. Sim. 2017. {\bf 20}(3) (2017), 148--162.

\bibitem{HiMa}
M. Hinz, J. Masamune, \emph{Removable sets and codimension four on manifolds and metric measure spaces}, in preparation.

\bibitem{Hi93}
F. Hirsch, \emph{Repr\'esentation du processus d'Ornstein-Uhlenbeck \`{a} $n$-param\`{e}tres}, S\'eminaire de probabilit\'es XXVII, Lect. Notes in Math. vol. {\bf 1557}, Springer, Berlin, 1993, 302--303. 

\bibitem{HiSo95}
F. Hirsch, S. Song, \emph{Markov properties of multiparameter processes and capacities},
Probab. Theory Relat. Fields {\bf 103} (1995), 45--71.

\bibitem{HiSo95Pota}
F. Hirsch, S. Song, \emph{Potential theory related to some multiparameter processes},
Pot. Anal. {\bf 4} (1995), 245--267.

\bibitem{IkWa89}
N. Ikeda, S. Watanabe, \emph{Stochastic Differential Equations and Diffusion Processes}, Second Edition,
North-Holland Publ., Amsterdam, 1989.

\bibitem{LMPnotes}
A. Lunardi, M. Miranda, D. Pallara, \emph{Infinite Dimensional Analysis}, 19th Internet Seminar 2015/2016: Infinite Dimensional Analysis, Dept. of Math. and Comp. Sci., University of Ferrara, URL: http://dmi.unife.it/it/ricerca-dmi/seminari/isem19/lectures.

\bibitem{M} Z. Ma, M. R\"ockner, \emph{Introduction to the Theory of (Non-symmetric) Dirichlet forms}, Springer-Verlag, 1992.

\bibitem{Ma97}
P. Malliavin, \emph{Stochastic Analysis}, Grundlehren der math. Wiss. Vol. 313, Springer, Berlin, 1997.

\bibitem{Masamune99}
J. Masamune, \emph{Essential self adjointness of Laplacians on Riemannian manifolds with fractal boundary}, 
Communications in Partial Differential Equations, {\bf 24}(3-4) (1999), 749-757.

\bibitem{Mazja72}
V. Maz'ya, \emph{Removable singularities of bounded solutions of quasilinear elliptic equations of any order}, J. Soviet Math. {\bf 3} (1975), 480-492.

\bibitem{Mazja}
V. Maz'ya, \emph{Sobolev Spaces}, Grundlehren der math. Wiss. 342, 2nd ed., Springer, New York, 2011.

\bibitem{MazjaKhavin}
V.G. Maz'ya, V.P. Khavin, \emph{Non-linear potential theory}, Russ. Math. Surveys vol. {\bf 27}(6) (1972), 71-148.

\bibitem{Na} 
R. Nagel, \emph{One-parameter Semigroups of Positive Operators}, Lect. Notes Math. vol. {\bf 1184}, Springer, Berlin, 1986.


\bibitem{R} 
M. Reed, B. Simon, \emph{Methods of Modern Mathematical Physics II - Fourier Analysis, Self-adjointness}, Academic Press, San Diego, 1975.

\bibitem{SII} 
I. Shigekawa, \emph{Sobolev spaces over the Wiener space based on an Ornstein-Uhlenbeck operator}, Mathematical Society, J. Math. Kyoto Univ. {\bf 32}(4) (1992), 731--748.

\bibitem{SI} 
I. Shigekawa, \emph{An example of regular $(r,p)$-capacity and essential self-adjointness of a diffusion operator in infinite dimensions}, J. Math. Kyoto Univ. {\bf 35}(4) (1995), 639--651.

\bibitem{S} 
I. Shigekawa, \emph{Stochastic Analysis}, Transl. Math. Monographs vol. 224, Amer. Math. Soc., Providence, 1998.

\bibitem{So93}
S. Song, \emph{Processus d'Ornstein-Uhlenbeck et ensembles
W2,2-polaires}, Pot. Anal. {\bf 2} (1993), 171--186.

\bibitem{So93b}
S. Song, \emph{In\'egalit\'es relatives aux processus d'Ornstein-Uhlenbeck \`{a} $n$ param\`{e}tres et capacit\'e
gaussienne $c_{n,2}$}, S\'eminaire de Probabilit\'es XXVII. Springer, Berlin, 1993, pp. 276--301.


\bibitem{St70}
E. Stein, \emph{Topics in Harmonic Analysis Related to the Littlewood-Paley Theory}, Annals of Mathematical Studies, Princeton University Press, Princeton, 1970.

\bibitem{Su88}
H. Sugita, \emph{Positive generalized Wiener functions and potential theory over abstract Wiener spaces}, Osaka J. Math. {\bf 25} (1988), 665--696.

\bibitem{Takeda92}
M. Takeda, \emph{The maximum Markovian self-adjoint extensions of generalized Schr\"odinger operators}, J. Math. Soc. Japan {\bf 44} (1) (1992), 113-130.


\end{thebibliography}
\end{document}